\documentclass[11pt,reqno]{amsart}

\usepackage{amsfonts}
\usepackage{graphicx}
\usepackage{amsmath} 
\usepackage{amsthm}
\usepackage{amssymb}
\usepackage{mathdots}
\usepackage{latexsym}
\usepackage{mathrsfs}
\usepackage{subfig}
\usepackage{color}
\usepackage{enumerate}
\usepackage{tikz}
\usepackage{multicol}
\usepackage[backref=page]{hyperref}
\usepackage{cite} 

\newtheorem{theorem}{Theorem}[section]
\newtheorem{lemma}[theorem]{Lemma}
\newtheorem{proposition}[theorem]{Proposition}

\newtheorem{definition}[theorem]{Definition}

\newcommand{\C}{\mathbb{C}}

\newcommand{\minimatrix}[4]{\begin{pmatrix} #1 & #2 \\ #3 & #4 \end{pmatrix}  }

\begin{document}

\title[Planar rook algebra with colors and Pascal's simplex]{Planar rook algebra with colors and Pascal's simplex}
\author[S. Mousley]{Sarah Mousley}
\author[N. Schley]{Nathan Schley}
\author[A. Shoemaker]{Amy Shoemaker}

\address{Utah State University 3900 Old Main Hill, Logan, UT 84322 \\ 
Department of Mathematics and Statistics}

\address{Department of Mathematics
South Hall, Room 6607
University of California
Santa Barbara, CA 93106-3080}

\address{Pomona College 610 North College Avenue, Claremont, CA 91711 \\
Department of Mathematics}


\begin{abstract}
We define $P_{n,c}$ to be the set of all diagrams consisting of two rows of $n$ vertices with edges, each colored with an element in a set of $c$ possible colors, connecting vertices in different rows. Each vertex can have at most one edge incident to it, and no edges of the same color can cross. 
In this paper, we find a complete set of irreducible representations of $\C P_{n,c}$. We show that the Bratteli diagram of $\C P_{0,c} \subseteq \C P_{1,c} \subseteq \C P_{2,c} \subseteq \cdots$ is Pascal's $(c+1)$-simplex, and use this to provide an alternative proof of the well-known recursive formula for multinomial coefficients.
\end{abstract}

\maketitle

\section{Introduction} \label{secIntroduction}

In this paper, we define a new algebra, the $c$ colored planar rook algebra, which we will denote by $\C P_{n,c}$. Our algebra $\C P_{n,c}$ is a unital, finite-dimensional vector space over $\C$ with an associative multiplication, the basis elements of which are diagrams with two rows of $n$ vertices and edges colored with at most $c$ colors. Our paper is a generalization of the results of Flath, Halverson, and Herbig in \cite{flath}, where all irreducible representations of the single colored planar rook algebra, $\C P_{n,1}$, are described explicitly and an alternative proof to the binomial recursion formula ${n \choose k}={n-1 \choose k-1} + {n-1 \choose k}$ is given. Similarly, through a consideration of the representation theory of our generalized algebra $\C P_{n,c}$, we will generate a proof of the recursive formula for multinomial coefficients that govern Pascal's $(c+1)$-simplex,
$$\binom{n}{n_0,n_1, \ldots, n_c}=\binom{n-1}{n_0-1,n_1, \ldots,n_c}+\binom{n-1}{n_0,n_1-1, \ldots, n_c}+\cdots +\binom{n-1}{n_0,n_1, \ldots, n_c-1}.$$

Now $\C P_{n,1}$ is a subalgebra of the partition algebra, a diagrammatic algebra defined independently by Jones \cite{jones} and Martin \cite{martin} to generalize the Temperley-Lieb algebra and the Potts model in statistical mechanics. Colorings of various subalgebras of the partition algebra have been investigated in recent years (see \cite{kennedy}, \cite{orellana}, \cite{parvathi}). A colored planar rook algebra like $\C P_{n,c}$, however, to the best of our knowledge, has yet to be explored. 

Our goal is to investigate the representation theory of $\C P_{n,c}$. Specifically, our main goals are the following:
\begin{itemize}
\item Find all irreducible representations of $\C P_{n,c}$ (Theorem \ref{main}).

\

\item Show that $\C P_{n,c}$ is isomorphic to a direct sum of matrix algebras (Theorem \ref{matrixiso}).

\
 
\item Show, through embedding $\C P_{n-1,c}$ into $\C P_{n,c}$, that the Bratteli diagram of the $\C P_{n,c}$ tower $\C P_{0,c} \subseteq \C P_{1,c} \subseteq \C P_{2,c} \subseteq \cdots$ is Pascal's $(c+1)$-simplex, a $(c+1)$-dimensional version of Pascal's triangle (Section \ref{section bratteli}). 

\

\item Show that the character of a diagram as a linear operator on an irreducible representation of $\C P_{n,c}$ is a product of $c$ binomial coefficients (Theorem \ref{character}).
\end{itemize}

\bigskip
\section{Defining the colored planar rook algebra, $\C P_{n,c}$} \label{section defining planar}

Throughout this paper, unless stated otherwise, we fix an integer $c$ such that $c \geq 1$. For $i \in \{1,\ldots, c\}$ we let $u_i$ denote the element in the ring $\mathbb{Z}_2^c$ with a 1 in the $i^{th}$ entry and zeros elsewhere, and we let ${0}$ denote the element in $\mathbb{Z}_2^c$ with zeros in every entry.   

Define $R_{0,c}= \{\emptyset\}$, and for $n \geq 1$ let $R_{n,c}$ denote the set of $n \times n$ matrices with entries from $\{{0},u_1,\ldots, u_c\}$ having at most one element from the set $\{u_1,\ldots,u_c\}$ in each row and in each column. For example, 
$\minimatrix{u_1}{0}{0}{u_2}, \minimatrix{0}{u_2}{0}{0} \in R_{2,2}$, but $\minimatrix{u_1}{u_2}{0}{0} \not \in R_{2,2}$. 
 We call these matrices \textit{rook matrices} since each $u_i$ can be thought of as a rook in a non-attacking position on an $n \times n$ chessboard.  

To each element of $\{u_1,\ldots,u_c\}$ we associate a unique color, and to each matrix in $R_{n,c}$ we associate a diagram with colored edges. If $M \in R_{n,c}$, then the corresponding diagram is the graph on two rows of $n$ vertices, such that vertex $i$ in the top row is connected to vertex $j$ in the bottom row by an edge with color $u_k$ if and only if $m_{ij}=u_k$. For example, in $R_{4,2}$ if $u_1$ is solid and $u_2$ is dotted, then

\begin{center}
\raisebox{-10pt}{\begin{tikzpicture}
[scale=.25,auto=center,every node/.style={fill, circle, inner sep=0pt,minimum size=5pt}]
  \node (n1) at (4,9)  {};
  \node (n2) at (4,6)  {};
  \node (n3) at (7,9) {};
  \node (n4) at (7,6)  {};
  \node (n5) at (10,9) {};
  \node (n6) at (10,6) {};
  \node(n7) at (13,9) {};
  \node (n8) at (13,6) {};
  
 \foreach \from/\to in {n3/n6}
    \draw [dotted, thick](\from) -- (\to);  
    
     \foreach \from/\to in {n1/n4,n7/n2}
    \draw [thick, color=black](\from) -- (\to);  
    
\end{tikzpicture}}
\ \ represents \ $\left( \begin{array}{cccc} {0} & u_1 & {0} & {0} \\
{0} & {0} & u_2 & {0} \\
{0} & {0} & {0} & {0} \\
u_1 & {0} & {0} & {0} \\
 \end{array} \right). $

\
\end{center}

Now that we have a diagrammatic description of $R_{n,c}$, we define a diagram multiplication.  For $d_1,d_2 \in R_{n,c}$ we define $d_1d_2$ by stacking $d_1$ on top of $d_2$ and fusing together the bottom vertices of $d_1$ with the top vertices of $d_2$. We let $d_1d_2$ be the diagram whose edges correspond to all monochromatic paths from the top of $d_1$ to the bottom of $d_2$. This diagrammatic multiplication is equivalent to computing the product of $d_1$ and $d_2$ as matrices. For example, in $R_{3,2}$ we having the following:
$$
\hspace{1 in}
\begin{tikzpicture}
 [scale=.15,auto=left,every node/.style={fill, circle, inner sep=0pt,minimum size=4.5pt}]
  \node (n1) at (4,18)  {\tiny $ $};
  \node (n2) at (4,14)  {\tiny $ $};
  \node (n3) at (7,18) {\tiny $ $};
  \node (n4) at (7,14)  {\tiny $ $};
  \node (n5) at (10,18) {\tiny $ $};
  \node (n6) at (10,14) {\tiny $ $};
    \node (m1) at (4,12)  {\tiny $ $};
  \node (m2) at (4,8)  {\tiny $ $};
  \node (m3) at (7,12) {\tiny $ $};
  \node (m4) at (7,8)  {\tiny $ $};
  \node (m5) at (10,12) {\tiny $ $};
  \node (m6) at (10,8) {\tiny $ $};

\foreach \from/\to in {m5/m2}
    \draw [dotted, thick, color=black](\from) -- (\to);    
\foreach \from/\to in {n1/n2, n3/n6,n5/n4, m1/m4}
    \draw [thick, color=black](\from) -- (\to);   

\end{tikzpicture}
\hspace*{-58pt}
\raisebox{33pt}{$d_1=$}
\hspace{58pt} 
\hspace*{-79pt}
\raisebox{10pt}{$d_2=$}
\hspace{79pt}
\hspace*{-40pt}
\raisebox{22pt}{ $=$  }
\hspace{40pt}
\hspace*{-40pt}
\raisebox{15pt}{
\begin{tikzpicture}
[scale=.15,auto=left,every node/.style={fill, circle, inner sep=0pt,minimum size=4.5pt}]
  \node (v1) at (4,6)  {\tiny $ $};
  \node (v2) at (4,2)  {\tiny $ $};
  \node (v3) at (7,6) {\tiny $ $};
  \node (v4) at (7,2)  {\tiny $ $};
  \node (v5) at (10,6) {\tiny $ $};
  \node (v6) at (10,2) {\tiny $ $};   
\foreach \from/\to in {v4/v1}
    \draw [thick, color=black](\from) -- (\to); 
 \end{tikzpicture}}
 \hspace{40pt}
 \hspace*{-35pt}
 \raisebox{22pt}{$= \ d_1d_2$}
 \hspace{35pt}
$$ 

$$d_1d_2=\left(\begin{array}{ccc} u_1 & 0 & 0  \\
0 & 0 & u_1 \\
0 & u_1 & 0 \\
\end{array} \right)
\left(\begin{array}{ccc} 0 & u_1 & 0 \\
0 & 0 & 0 \\
u_2 & 0 & 0 \\
\end{array} \right)=
\left(\begin{array}{ccc} 0 & u_1 & 0 \\
 0 & 0 & 0 \\
0 & 0 & 0 \\
\end{array} \right).$$

\noindent Observe that $R_{n,c}$ is a semigroup, and does not contain an identity diagram unless $c=1$ or $n=0$. 

We say an element in $R_{n,c}$ is \textit{planar} if its diagram can be drawn (keeping inside of the rectangle formed by its vertices) such that no edges of the same color cross. Define $P_{0,c}=\{\emptyset\}$, and for $n \geq 1$ define $P_{n,c}$ to be the set of all planar diagrams in $R_{n,c}$. For example, $d_2$ and $d_1d_2$ in the example above are planar, but $d_1$ is not, as two solid edges cross each other. Diagram multiplication shows the product of planar diagrams is planar, and thus $P_{n,c}$ is a sub-semigroup of $R_{n,c}$. 

Throughout this paper, we are concerned only with planar diagrams. Considering the semigroup $P_{n,c}$ will allow us to make a combinatorial connection between rook diagrams and Pascal's simplex in Section \ref{section bratteli}. For now, we note the cardinality of $P_{n,c}$.

\begin{proposition}\label{pnc card} The number of diagrams in $P_{n,c}$ is given by
 $$|P_{n,c}|=\sum_{n_0+\cdots + n_c=n}\binom{n}{n_0, \cdots,  n_c}^2,$$ 
 \noindent where each $n_i$ is a non-negative integer.
\end{proposition} 

\begin{proof}
Let $n_0,\ldots, n_c$ be non-negative integers such that $n_0 + \cdots + n_c=n$. Suppose we want to create a diagram $d \in P_{n,c}$ with exactly $n_i$ edges colored $u_i$ for all $i \in \{1,\ldots, c\}$. Then we can choose $n_0$ vertices in the top row of $d$ to be isolated and $n_i$ to be incident to edges colored $u_i$ for all $i \in \{1,\ldots, c\}$ in $\binom{n}{n_0, \cdots, n_c}$ ways. We can make a similar choice in the bottom row of $d$ in $\binom{n}{n_0, \cdots , n_c}$ ways. Once these choices have been made, there is exactly one planar way to connect the vertices. Thus, $|P_{n,c}|=\sum\limits_{n_0+\cdots + n_c=n}\binom{n}{n_0, \cdots, n_c}^2.$ 
\end{proof} 

 We now construct an algebra, $\C P_{n,c}$, associated to the semigroup $P_{n,c}$ by taking formal linear combinations of elements in $P_{n,c}$. 

\begin{definition}
The planar rook algebra $\C P_{n,c}$ is the $\C$-vector space whose basis is $P_{n,c}$. That is, 
$$\C P_{n,c}=\C\text{-span}\{d~|~d\in P_{n,c}\}=\left\{\sum_{d\in P_{n,c}}\lambda_dd~|~\lambda_d\in\C\right\},$$
where linearly extending the multiplication of $P_{n,c}$ equips $\C P_{n,c}$ with multiplication.
\end{definition}

Since the basis of $\C P_{n,c}$ is $P_{n,c}$, the algebra's dimension is equal to $\sum\limits_{n_0+\cdots + n_c=n}\binom{n}{n_0, \cdots, n_c}^2$ by Proposition \ref{pnc card}. Let us also note the distinction between the empty diagram (associated to the zero matrix),
\raisebox{-2.5pt}{$\begin{tikzpicture}
  [scale=.1,auto=left,every node/.style={fill, circle, inner sep=0 pt,minimum size=4pt}]
  \node (n1) at (4,9)  {\tiny $ $};
  \node (n2) at (4,6)  {\tiny $ $};
  \node (n3) at (7,9) {\tiny $ $};
  \node (n4) at (7,6)  {\tiny $ $};
  \node (n5) at (10,9) {\tiny $ $};
  \node (n6) at (10,6) {\tiny $ $};
  \node (n7) at (13,9) {\tiny $ $};
  \node(n8) at (13,6) {\tiny $ $};
  \node (n9) at (16,9) {\tiny $ $};
  \node (n10) at (16,6) {};
 \end{tikzpicture}$},
 and the zero vector ${\bf 0}$. The former is a basis element of $\C P_{n,c}$, while ${\bf 0}$ is the zero linear combination of all elements of $P_{n,c}$.
  
We note that even though in general $P_{n,c}$ does not have an identity element, $\C P_{n,c}$ is unital. For $i\in \{1,\ldots , c\}$ define $I_i$ to be a single vertical edge colored $u_i$, and let $I_0=\raisebox{-6pt}{\begin{tikzpicture}
  [scale=.15,auto=left,every node/.style={circle,fill=black!100,inner sep=0pt, minimum size=4.5pt}]
  \node (n1) at (4,9)  {};
  \node (n2) at (4,6)  {\tiny $ $};
   \foreach \from/\to in {}
    \draw [thick, color=black](\from) -- (\to);
    \end{tikzpicture}}$ be a single pair of isolated vertices.
Then the identity element of $\C P_{1,c}$, denoted by $e_c$, is given by $$e_c=\left(\sum \limits_{i=1}^c I_i\right)-(c-1)I_0.$$ For $d_1 \in P_{n,c}$ and $d_2 \in P_{m,c}$, we define $d_1 \otimes d_2$ to be the diagram resulting from concatenating $d_2$ to the right of $d_1$. We extend this operation linearly to define $g_1 \otimes g_2$ for $g_1 \in \C P_{n,c}$ and $g_2 \in C P_{m,c}$. Using this notation, the identity element of $\C P_{n,c}$ is $$e_c^{\otimes n}= \underbrace{e_c \otimes e_c \otimes \cdots \otimes e_c}_{n}.$$

\bigskip

\section{Irreducible Representations of the Planar Rook Algebra}
In this paper, by an \textit{algebra}, we shall mean a unital, finite-dimensional  associative algebra over $\C$. A \textit{representation} or 
\textit{module} of an algebra $A$ is a vector space $V$ together with an action of $A$ on
$V$ by linear automorphisms. 
An $A$-module $V$ is said to be \textit{irreducible} if the only $A$-submodules of $V$ 
(equivalently, $A$-invariant subspaces of $V$) are trivial or improper. For introductory 
representation theory definitions as these, see \cite{barcelo}.
 
The \textit{regular representation} of $\C P_{n,c}$ is when $\C P_{n,c}$ acts on itself by left multiplication. In this section, we show that the regular representation of $\C P_{n,c}$ decomposes into a direct sum of irreducible $\C P_{n,c}$-modules. From here, we apply the Artin-Weddernburn theory of semisimple algebras, stated below as Theorem \ref{artin} (see \cite{curtis} chp. 4), to give an explicit list of all irreducible $\C P_{n,c}$-modules up to isomorphism (Theorem \ref{main}). 

\begin{theorem}[Artin-Wedderburn] \label{artin}
If the regular representation of an algebra $A$ decomposes into a direct sum of irreducible $A$-modules, then 
\begin{enumerate}[(a)]
\item $A$ is semisimple (every $A$-module can be decomposed into a direct sum of irreducible $A$-modules), and 
\item every irreducible $A$-module is isomorphic to a direct summand of the regular representation of $A$.
\end{enumerate}
\end{theorem}

We note that a subspace $V$ of $\C P_{n,c}$ is $P_{n,c}$-invariant under the action of left multiplication if and only if $V$ is $\C P_{n,c}$-invariant. Thus, in what follows 
we seek to find $P_{n,c}$-invariant, irreducible subspaces of $\C P_{n,c}$ that decompose $\C P_{n,c}$. 

A natural candidate to consider is $Y^n_k$, the subspace of $\C P_{n,c}$ spanned by all elements in $P_{n,c}$ with $k$ or fewer edges. For $d \in P_{n,c}$, let size$(d)$ be the number of edges in $d$. Now diagram multiplication shows that for all $d_1,d_2 \in P_{n,c}$ 
\begin{equation} \label{sizemin}
\text{size}(d_1d_2) \leq \text{min} \{\text{size}(d_1),\text{size}(d_2)\}.
\end{equation}

Thus $Y^n_k$ is $P_{n,c}$-invariant; however, $Y^n_0 \subseteq Y^n_1 \subseteq Y^n_2 \subseteq \cdots$, so the $Y^n_k$ do not decompose $\C P_{n,c}$.  Thus, we consider a change of basis that will lead us to a decomposition of $\C P_{n,c}$ into $P_{n,c}$-invariant subspaces.

Let $d_1,d_2 \in P_{n,c}$. We say that $d_1 \subseteq d_2$ if all edges in $d_1$ are contained in $d_2$ in the same color. If $d_1 \subseteq d_2$, we let size$(d_2,d_1) = \text{size}(d_2)-\text{size}(d_1)$. For each $d \in P_{n,c}$, define

\[x_d=\sum \limits_{d^\prime \subseteq d} (-1)^{\text{size}(d,d')}d^\prime.\] 

It can be verified that $d=\sum \limits_{d' \subseteq d} x_{d'}$, and since $|\{x_d~|~ d \in P_{n,c}\}|=|P_{n,c}|$, we see that $\{x_d ~|~d \in P_{n,c}\}$ is a basis for $\C P_{n,c}$ over $\C$.  
We will describe the action of $\C P_{n,c}$ on our new basis $\{x_d~|~ d \in P_{n,c}\}$ (Proposition \ref{3.1eq}), but first we must develop some notation. 

For a planar rook diagram $d\in P_{n,c}$, let $\tau(d)=(\tau_0(d),\ldots, \tau_c(d))$, where  $\tau_0(d)$ is the set of all isolated vertices in the top row of $d$ and $\tau_i(d)$ is the set of vertices in the top row of $d$ incident to an edge colored $u_i$ for $i \in \{1,\ldots,c\}$. We define $\beta(d)=(\beta_0(d), \ldots, \beta_c(d))$ similarly to describe the bottom of vertices of $d$.
 
For example, if $u_1$ is solid, $u_2$ is dotted, and
$d=\raisebox{-5pt}{\begin{tikzpicture}
  [scale=.15,auto=left,every node/.style={fill, circle, inner sep=0pt,minimum size=4pt}]
  \node (n1) at (4,9)  {\tiny $ $};
  \node (n2) at (4,6)  {\tiny $ $};
  \node (n3) at (7,9) {\tiny $ $};
  \node (n4) at (7,6)  {\tiny $ $};
  \node (n5) at (10,9) {\tiny $ $};
  \node (n6) at (10,6) {\tiny $ $};
  \node (n7) at (13,9) {\tiny $ $};
  \node(n8) at (13,6) {\tiny $ $};
  \node (n9) at (16,9) {\tiny $ $};
  \node (n10) at (16,6) {\tiny $ $};

   \foreach \from/\to in {n1/n4, n5/n6}
    \draw [dotted, thick, color=black](\from) -- (\to);
    
\foreach \from/\to in {n9/n10,n2/n3}
    \draw [thick, color=black](\from) -- (\to); 
\end{tikzpicture}} \in P_{5,2}$,
 then $\tau(d)=(\{4\},\{2,5\},\{1,3\})$ and $\beta(d)=(\{4\},\{1,5\},\{2,3\})$. 
Note that given the information stored in these tuples, there is exactly one planar way to connect the vertices in $d$; thus, $\tau(d)$ and $\beta(d)$ completely determine $d$.
We now have the definitions necessary to proceed. 

\begin{proposition} \label{3.1eq}
Let $a,d \in P_{n,c}$. Then 

\[dx_a=\begin{cases}
  x_{da}, & \text{if } \tau_i(a)\subseteq \beta_i(d)\text{ for all } i \in \{1,\ldots, c\} \\
   {\bf 0},     & \text{otherwise.} 
  \end{cases}\]
\end{proposition}   

\begin{proof}
Suppose $\tau_{i}(a)\subseteq \beta_{i}(d)$ for all $i \in \{1,\ldots, c\}$. Then for all $a^\prime \subseteq a,$ we have $\beta(da')=\beta(a')$. Thus, size$(a, a') =$ size$(da, da')$. From this we obtain
 
\begin{align*}
dx_a= \sum_{a^\prime \subseteq a} (-1)^{\text{size}(a, a^\prime)} da^\prime =  \sum_{a' \subseteq a} (-1)^{\text{size}(da, da')} da'  =  \sum_{(da)' \subseteq da} (-1)^{\text{size}(da,(da)')} (da)'  = x_{da}.
\end{align*} 

Now suppose there exists $k \in \tau_{i}(a) \setminus \beta_{i}(d)$ for some $i \in \{1,\ldots,c \}$. Let $A_1 = \{a'~|~ a' \subseteq a, \text{ and } k \in \tau_{i}(a')\}$ and $A_2 = \{a' ~|~ a' \subseteq a, \text{ and } k \notin \tau_{i}(a')\}$. Consider the bijection $\phi:A_1 \rightarrow A_2$, where $\phi(a')$ is the diagram resulting from the removal of the edge incident to the $k^{th}$ top vertex of $a'$. Now for all $a' \in A_1$ we know $d a' = d \phi (a')$  and size$(a, a')= $ size$(a, \phi(a')) +1$. Thus we have

\begin{align*}
dx_a= & \sum_{a' \in A_1} (-1)^{\text{size}(a, a')} da' + \sum_{a' \in A_2} (-1)^{\text{size}(a, a')} da'\\
=& \sum_{a' \in A_1} (-1)^{\text{size}(a, \phi(a'))+1} da' + \sum_{a' \in A_1} (-1)^{\text{size}(a, \phi(a'))} d\phi(a') \\
=& \ { \bf 0}.
\end{align*} 

\end{proof}

We now begin to search for irreducible, $P_{n,c}$-invariant subspaces of $\C P_{n,c}$. First, consider
 
\[W^{n,k}=\mathbb{C}\text{-span}\{x_a \ | \ a \in P_{n,c} \text{ and size}(a)=k\}.\]

\

Now for all $a,d \in P_{n,c}$, it is easy to show that $\tau_i(a) \subseteq \beta_i(d)$ for all $i \in \{1,\ldots, c\}$ if and only if size$(a)$ = size$(da)$. This fact, in unison with Proposition \ref{3.1eq}, shows that the $W^{n,k}$ are $P_{n,c}$-invariant. However, the $W^{n,k}$ are reducible, which we shall see by considering the following subspaces of the $W^{n,k}$. 

Throughout, we let 
$$X_n=\{\beta(d)~|~d \in P_{n,c}\},$$ 
\noindent and for $T\in X_n$ we define 

\[W^n_T = \mathbb{C}\text{-span}\{x_a  \ | \ a \in P_{n,c}, \text{ and } \beta(a)=T\}.\]

\

Let $a,d \in P_{n,c}$. Now if $\tau_i(a) \subseteq \beta_i(d)$ for all $i \in \{1,\ldots, c\}$, then $\beta(a)= \beta(da)$ and by Proposition \ref{3.1eq}, we have $dx_a=x_{da}$. If $\tau_i(a) \not \subseteq \beta_i(d)$ for some $i \in \{1, \ldots, c\}$, then $dx_a={\bf 0}$. Thus, $W^n_T$ is a $P_{n,c}$-invariant subspace of $W^{n,k}$, where $k=\sum \limits_{i=1}^c |T_i|$, and thus a $\C P_{n,c}$-module. The following two theorems describe the structure of the $W^n_T$.

\begin{theorem}\label{irred}
Let $T\in X_{n}$. Then $W^{n}_T$ is irreducible (it contains no proper, nonzero $P_{n,c}$-invariant subspaces).
\end{theorem}

\begin{proof}

Suppose $V$ is a nonzero $P_{n,c}$-invariant subspace of $W^n_T$. Choose ${\bf w}\in V$ such that ${\bf w}\neq {\bf 0}$, and expand ${\bf w}$ in the basis of $W^{n}_T$ as  
${\bf w}=\sum\limits_{x_a \in W^n_T}\lambda_ax_a$, where $\lambda_a\in \C$. Now ${\bf w}\neq {\bf 0}$ reveals that $\lambda_a\neq 0$ for some $a$. Let $a'$ be the unique planar diagram with $\tau(a')=\beta(a')=\tau(a)$. By assumption, $V$ is $P_{n,c}$-invariant, so $a'{\bf w}=a'\lambda_{a}x_{a}=\lambda_{a}x_a\in V$, from which we conclude that $x_{a}\in V.$

Now fix arbitrary an element $x_d \in W^n_T$. Let $d'$ be the unique planar diagram with $\beta(d')=\tau(a)$ and $\tau(d')=\tau(d)$. Then $d'x_{a}=x_{d'a}=x_{d}\in V$, since $V$ is $P_{n,c}$-invariant. This shows that all basis elements of $W^{n}_T$ are in $V$. Thus, every nonzero $P_{n,c}$-invariant subspace of $W^n_T$ is $W^n_T$ itself. 

\end{proof}

\begin{theorem} \label{isomorphic}
Let $T,S \in X_n$. Then $W^n_T$ and $W^n_{S}$ are isomorphic $\C P_{n,c}$-modules if and only if $|T_i|=|S_i|$ for all $i \in \{0,\ldots,c\}$.
\end{theorem}

\begin{proof}

Let $T, S \in X_{n}$. First suppose that $|T_i|=|S_i|$ for all $i \in \{0,\ldots,c\}$. Consider the bijection $\phi: W^n_T \rightarrow W^n_S$ given by $\phi(x_a)=x_{a}d$, where $d \in P_{n,c}$  denotes the unique diagram with $\tau(d)=T$ and $\beta(d)=S$. Now fix $d' \in P_{n,c}$. Then,
$$\phi(d'x_a)=(d'x_{a})d=d'(x_ad)=d'\phi(x_a).$$ 
\noindent Thus, we see that $W^n_T \cong W^n_{S}$.

Now suppose $|T_i| \not = |S_i|$ for some $i \in \{1,\ldots, c\}$. Without loss of generality, assume $|T_i|<|S_i|$.  
Consider an arbitrary $x_a \in W^n_T$. Let $d' \in P_{n,c}$ be the unique diagram with $\beta(d')=\tau(d')=\tau(a)$. Then $\tau_j(a)
\subseteq \beta_j(d')$ for all $j \in \{1,\ldots, c\}$. So by Proposition \ref{3.1eq}, $d'x_a=x_{d'a}=x_{a}\not = {\bf 0}$.

However, for all $x_{a'} \in W^{n}_S$, we have $\beta(a')=S$, which implies $ |\beta_i(d')|=|T_i| <|S_i|=|\tau_i(a')|$. It follows that $\tau_i(a') \not \subseteq \beta_i(d')$, and thus by Proposition \ref{3.1eq}, $d'x_{a'}={\bf 0}$. Thus $d'$ acts as the zero map on $W^n_S$, but acts non-trivially on $W^n_T$. Therefore $W^{n}_T$ and $W^{n}_S$ cannot possibly be isomorphic.
\end{proof}

Since we desire a complete set of irreducible representations of $\C P_{n,c}$ up to isomorphism, by Theorem \ref{isomorphic} we make the following definition.
\begin{definition} \label{defisoclass}
For non-negative integers $n_0,\ldots,n_c$ such that $n_0 + \cdots + n_c=n$, let $W^n_{n_0,\ldots , n_c}$ be a representative of the isomorphism class of all $W^{n}_T$ with $|T_i|=n_i$ for all $i\in \{0,\ldots ,c\}$ as $\C P_{n,c}$-modules.
\end{definition}

We now are ready to prove one of our main results. Below we give a decomposition of $\C P_{n,c}$ into irreducible $\C P_{n,c}$-modules and give an explicit description of all $\C P_{n,c}$-modules.

\begin{theorem} \label{main}
\noindent \begin{enumerate} 
 \item The decomposition of $\mathbb{C}P_{n,c}$ into irreducible $\C P_{n,c}$-modules is given by
\begin{align*}
\mathbb{C}P_{n,c} &= \bigoplus \limits_{k=0}^{n}  W^{n,k}= \bigoplus \limits_{k=0}^{n} \bigoplus \limits_{~~~~T \in X_{n,k}}  W^{n}_T \cong \bigoplus\limits_{n_0+\cdots +n_c=n}\binom{n}{n_0,\ldots , n_c}W^n_{n_0,\ldots , n_c},
\end{align*}
\noindent where $X_{n,k}=\{T~|~ T \in X_n, \text{ and } \sum \limits_{i=1}^c |T_i|=k\}$. 
\item The set $\{W^{n}_{n_0,\ldots,n_c} ~| ~n_0+\cdots+n_c=n \}$ is a complete set of irreducible $\mathbb{C}P_{n,c}$-modules.
\end{enumerate}
\end{theorem}

\begin{proof}
The first equality in part (1) follows from the fact that each element in the basis $\{x_a~|~ a \in P_{n,c}\}$ of $\C P_{n,c}$ is contained in exactly one $W^{n,k}$, and the second equality follows from the fact that a given $x_a$ is contained in exactly one $W^{n}_T$. Finally, the isomorphism follows because $W^n_{T} \cong W^n_{|T_0|,\ldots,|T_c|}$ (see Theorem \ref{isomorphic}) and the number of $T \in X_n$ such that $|T_i|=n_i$ for all $i \in \{0,\ldots, c\}$ is ${n \choose n_0, \ldots, n_c}$. 

Part (2) follows directly from part (1) and Theorem \ref{artin}, the Artin-Weddernburn theory of semisimple algebras. 
\end{proof}

In the remainder of this section, we will further describe the structure of $\C P_{n,c}$ by showing that it is isomorphic to a direct sum of matrix algebras (Theorem \ref{matrixiso}). To do so, we require the following proposition, the proof of which we omit as it is similar to Proposition \ref{3.1eq}. 

\begin{proposition} \label{3.1eqcor} Let $a,d \in P_{n,c}$. Then 
\[x_ad=\begin{cases}
x_{ad}, & \text{if } \beta_{i}(a)\subseteq \tau_{i}(d)\text{ for all } i
\in \{1,\ldots, c\} \\
{\bf 0}, & \text{otherwise.}
\end{cases}\] 
\end{proposition}

The proposition above describes how the basis $\{x_d~ |~ d \in P_{n,c}\}$ of $\C P_{n,c}$ acts on the basis $\{d ~|~ d \in P_{n,c}\}$ of $\C P_{n,c}$. In what follows, we utilize Propositions \ref{3.1eq} and \ref{3.1eqcor} to describe the action from the perspective of the basis $\{x_d ~|~ d \in P_{n,c}\}$ acting on itself. 
First we make the following definition.

Let $S,T \in X_{n}$ such that $|S_{i}|=|T_{i}|$ for all $i \in \{0,\ldots, c\}$. Define $$x_{S,T}=x_d,$$ where $d$ is the unique planar diagram with $\tau(d)=S \text{ and } \beta(d)=T.$

\begin{proposition} \label{3.3eq}
Let $S,T,U,V \in X_n$ such that $|S_i|=|T_i|$ and $|U_i|=|V_i|$ for all $i\in \{0, \ldots , c\}$. Then
$$x_{S,T}x_{U,V}=\begin{cases}
x_{S,V},& \text{ if } T=U\\
{\bf 0}, &\text{ if } T\neq U.
\end{cases}$$
\end{proposition}

\begin{proof}
Let $d,a\in P_{n,c}$ be the diagrams such that $x_{S,T}=x_d$ and $x_{U,V}=x_a$.

First, suppose $T= U$. By Proposition \ref{3.1eq}, we have $dx_a=x_{da}$. Now let $d'$ be a diagram such that $d'\subset d$. Then $ \tau_{i}(a)=\beta_{i}(d) \not\subseteq \beta_i(d')$ for some $i \in \{1,\ldots, c\}$. Thus $d'x_a={\bf 0}$, and we see
\[x_dx_a=dx_a+\sum_{d'\subset d}(-1)^{\text{size}(d, d')}d'x_a =x_{da}.\]
Since $\tau(da)=S$ and $\beta(da)=V$, we conclude that $x_{S,T}x_{U,V}=x_{S,V}$.

Next suppose $T\neq U$. Then for some $i \in \{1,\ldots, c\}$, it must be that $U_{i} \not\subseteq T_{i}$ or $T_{i} \not\subseteq U_{i}$. If $U_{i}\not\subseteq T_{i}$, then because $\beta_{i}(d^\prime) \subseteq \beta_{i}(d)$ for all $d^\prime \subseteq d$, clearly $\tau_{i}(a) \not \subseteq \beta_{i}(d^\prime)$. Thus, by Proposition \ref{3.1eq}, for all $d^\prime \subseteq d$ we have $d^\prime x_a={\bf 0}$. Therefore $x_dx_a={\bf 0}$, as desired. Similarly, if $T_{i} \not\subseteq U_{i}$, then since $\tau_{i}(a^\prime) \subseteq \tau_{i}(a)$ for all $a^\prime \subseteq a$, we have $\beta_{i}(d) \not \subseteq \tau_{i}(a^\prime)$. It follows from Proposition \ref{3.1eqcor} that 
$x_d a^\prime={\bf 0}$.
Thus, $x_dx_a={\bf 0}$, and we are done. 
\end{proof}

Additionally, we require the following lemma from \cite{curtis} (see chp. 4), restated in terms of terminology used in this paper.  

\begin{lemma} \label{twosided}
Let $I^n_{n_0,\ldots, n_c}=\C$-span$\{x_d ~|~ d \in P_{n,c}, \text{ and }|\beta_i(d)|=n_i \text{ for all } i \in \{0,\ldots, c\}\}$.
Then $I^n_{n_0,\ldots,n_c}$ is an algebra and a two-sided ideal of $\C P_{n,c}$. Moreover, $\C P_{n,c}$ is a direct sum (as an algebra) of all the ideals $I^n_{n_0,\ldots,n_c}$ obtained by letting $n_0,\ldots, n_c$ range over all possible values such that $n_0+\cdots + n_c=n$.
\end{lemma}

Finally, we are ready to prove one of our main results. By comparing the behavior of the $x_{S,T}$ to that of elementary matrices, we find the following structure of $\C P_{n,c}$:

\begin{theorem}\label{matrixiso}
\[\C P_{n,c} \cong \bigoplus \limits_{n_0+\cdots + n_c =n} \text{Mat}\left( {n \choose n_0, \ldots , n_c}, {n \choose n_0 ,\ldots ,n_c} \right),\]

\noindent where Mat$(m,m)$ is the algebra of all $m \times m$ matrices with complex entries.
\end{theorem}

\begin{proof}
By Lemma \ref{twosided}, it follows that $\C P_{n,c} = \bigoplus \limits_{n_0+\cdots +n_c=n}I^n_{n_0,\ldots,n_c}$ as an algebra. 

Fix $n_0,\ldots , n_c$. We now will show that $I^n_{n_0,\ldots, n_c}$ and Mat$\left( {n \choose n_0,\ldots,n_c},{n \choose n_0,\ldots,n_c}\right)$ are isomorphic algebras. 

To each element $S \in X_n$ such that $|S_i|=n_i$ for all $i \in \{0,\ldots, c\}$, assign a unique number in $\{1,\ldots , \binom{n}{n_0,\ldots , n_c}\}$. Now consider the bijection $\phi: I_{n_0,\ldots , n_c}\rightarrow \text{ Mat} \left( \binom{n}{n_0,\ldots, n_c},\binom{n}{n_0,\ldots , n_c} \right) $ given by $\phi(x_{S,T})=E_{i,j}$, where $i$ and $j$ are the index values of $S$ and $T$ respectively and where $E_{i,j}$ is the elementary matrix of $\text{ Mat}\left( \binom{n}{n_0,\ldots, n_c},\binom{n}{n_0,\ldots , n_c}\right)$ with a $1$ in the $(i,j)$-entry and zeros elsewhere. 

Now choose $x_{S,T},x_{U,V}\in I^n_{n_0,\ldots,n_c}$ arbitrarily. First suppose that $T=U$. Then by Proposition \ref{3.3eq}, $\phi(x_{S,T}x_{U,V})=\phi(x_{S,V})=E_{i,k}$. Similarly, $\phi(x_{S,T})\phi(x_{U,V})=E_{i,j}E_{j,k}=E_{i,k}$, by properties of multiplication of elementary matrices.

Next suppose $T\neq U$. Then $\phi(x_{S,T}x_{U,V})=\phi({\bf 0})=\left(\begin{array}{ccc} 0& \cdots & 0 \\
\vdots & \ddots & \vdots \\
0 & \cdots & 0 \\
\end{array} \right) 
=E_{i,j}E_{\ell,k}=\phi(x_{S,T})\phi(x_{U,V})$. Thus, $I_{n_0,\ldots , n_c}$ and $\text{ Mat}\left(\binom{n}{n_0,\ldots, n_c},\binom{n}{n_0,\ldots , n_c}\right)$ are isomorphic algebras. It follows that $\C P_{n,c} = \bigoplus\limits_{n_0+\cdots +n_c=n} I_{n_0,\ldots , n_c}\cong \bigoplus \limits_{n_0+\cdots + n_c =n} \text{Mat}\left( {n \choose n_0, \ldots , n_c}, {n \choose n_0 ,\ldots ,n_c} \right)$, as desired.

\end{proof}

\bigskip


\section{The Bratteli Diagram and Pascal's Simplex}\label{section bratteli}

Throughout this section, we fix $n \geq 1$. We will embed $\C P_{n-1,c}$ into $\C P_{n,c}$ to create a tower of algebras, $\C P_{0,c} \subseteq \C P_{1,c} \subseteq \C P_{2,c} \subseteq \cdots$. We will then create the corresponding Bratteli diagram and make a connection to a common combinatorial identity for multinomial coefficients.

Let $i\in \{0,\ldots , c\}$. For $i \not =0$ define $I_i$ to be a single vertical edge colored $u_i$, and let $I_0=\raisebox{-6pt}{\begin{tikzpicture}
  [scale=.15,auto=left,every node/.style={circle,fill=black!100,inner sep=0pt, minimum size=4.5pt}]
  \node (n1) at (4,9)  {};
  \node (n2) at (4,6)  {\tiny $ $};
   \foreach \from/\to in {}
    \draw [thick, color=black](\from) -- (\to);
    \end{tikzpicture}}$ be a single pair of isolated vertices. Then for $d\in P_{n-1,c}$, let $d\otimes I_i$ be the diagram in $P_{n,c}$ with $I_i$ concatenated to the right end of $d$. Extending this operation, for $g \in \C P_{n-1,c}$, let $g\otimes I_i$ be the diagram in $\C P_{n,c}$ with $I_i$ concatenated to the right end of each term in $g$.
We can use this concatenation to embed $\C P_{n-1,c}$ into $\C P_{n,c}$ via the embedding function $\gamma : \C P_{n-1,c}\rightarrow \C P_{n,c}$ given by 

\begin{equation} \label{embfunc}
\gamma(g)= \left(\sum_{i=1}^c g\otimes I_i\right) -(c-1)(g \otimes I_0).
\end{equation} 

For example, we embed an element $\C P_{2,2}$ into $\C P_{3,2}$ as follows:
\vspace{-1pt}
$$
\raisebox{-5pt}{\begin{tikzpicture}
  [scale=.15,auto=left,every node/.style={circle,fill=black!100,inner sep=0pt, minimum size=4.5pt}]
  \node (n1) at (4,9)  {};
  \node (n2) at (4,6)  {\tiny $ $};
  \node (n3) at (7,9) {\tiny $ $};
  \node (n4) at (7,6)  {\tiny $ $};
  \foreach \from/\to in {n2/n3}
    \draw [thick, color=black](\from) -- (\to);
\foreach \from/\to in {n1/n4}
    \draw [dotted,thick, color=black](\from) -- (\to);
    \end{tikzpicture}}
  \raisebox{0pt}{ $+$ 5}
\raisebox{-5pt}{\begin{tikzpicture}
  [scale=.15,auto=left,every node/.style={circle,fill=black!100,inner sep=0pt, minimum size=4.5pt}]
  \node (n1) at (4,9)  {};
  \node (n2) at (4,6)  {\tiny $ $};
  \node (n3) at (7,9) {\tiny $ $};
  \node (n4) at (7,6)  {\tiny $ $};
  \foreach \from/\to in {n2/n1}
    \draw [thick, color=black](\from) -- (\to);
    \end{tikzpicture}}
\raisebox{0pt}{   $~\mapsto~$ }
\left(
\raisebox{-5pt}{\begin{tikzpicture}
  [scale=.15,auto=left,every node/.style={circle,fill=black!100,inner sep=0pt, minimum size=4.5pt}]
  \node (n1) at (4,9)  {};
  \node (n2) at (4,6)  {\tiny $ $};
  \node (n3) at (7,9) {\tiny $ $};
  \node (n4) at (7,6)  {\tiny $ $};
  \node (n5) at (10,9) {\tiny $ $};
  \node (n6) at (10,6) {\tiny $ $};
 \foreach \from/\to in {n2/n3,n5/n6}
    \draw [thick, color=black](\from) -- (\to);
\foreach \from/\to in {n1/n4}
    \draw [dotted, thick,color=black](\from) -- (\to);
    \end{tikzpicture}}
\raisebox{0pt}{ $+$ 5}
\raisebox{-5pt}{\begin{tikzpicture}
  [scale=.15,auto=left,every node/.style={circle,fill=black!100,inner sep=0pt, minimum size=4.5pt}]
  \node (n1) at (4,9)  {};
  \node (n2) at (4,6)  {\tiny $ $};
  \node (n3) at (7,9) {\tiny $ $};
  \node (n4) at (7,6)  {\tiny $ $};
  \node (n5) at (10,9) {\tiny $ $};
  \node (n6) at (10,6) {\tiny $ $};
  \foreach \from/\to in {n2/n1,n5/n6}
    \draw [thick, color=black](\from) -- (\to);
    \end{tikzpicture}}
\right)
\raisebox{0pt}{ $+$ }
\left(
\raisebox{-5pt}{\begin{tikzpicture}
  [scale=.15,auto=left,every node/.style={circle,fill=black!100,inner sep=0pt, minimum size=4.5pt}]
  \node (n1) at (4,9)  {};
  \node (n2) at (4,6)  {\tiny $ $};
  \node (n3) at (7,9) {\tiny $ $};
  \node (n4) at (7,6)  {\tiny $ $};
  \node (n5) at (10,9) {\tiny $ $};
  \node (n6) at (10,6) {\tiny $ $};
 \foreach \from/\to in {n2/n3}
    \draw [thick, color=black](\from) -- (\to);
\foreach \from/\to in {n1/n4,n5/n6}
    \draw [dotted, thick,color=black](\from) -- (\to);
    \end{tikzpicture}}
    \raisebox{0pt}{ $+$ 5}
   \raisebox{-5pt}{ \begin{tikzpicture}
  [scale=.15,auto=left,every node/.style={circle,fill=black!100,inner sep=0pt, minimum size=4.5pt}]
  \node (n1) at (4,9)  {};
  \node (n2) at (4,6)  {\tiny $ $};
  \node (n3) at (7,9) {\tiny $ $};
  \node (n4) at (7,6)  {\tiny $ $};
  \node (n5) at (10,9) {\tiny $ $};
  \node (n6) at (10,6) {\tiny $ $};
  \foreach \from/\to in {n2/n1}
    \draw [thick, color=black](\from) -- (\to);
    \foreach \from/\to in {n5/n6}
    \draw [dotted, thick, color=black](\from) -- (\to);
    \end{tikzpicture}}
\right)
\raisebox{0pt}{ $-1$}
\left(
\raisebox{-5pt}{\begin{tikzpicture}
  [scale=.15,auto=left,every node/.style={circle,fill=black!100,inner sep=0pt, minimum size=4.5pt}]
  \node (n1) at (4,9)  {};
  \node (n2) at (4,6)  {\tiny $ $};
  \node (n3) at (7,9) {\tiny $ $};
  \node (n4) at (7,6)  {\tiny $ $};
  \node (n5) at (10,9) {\tiny $ $};
  \node (n6) at (10,6) {\tiny $ $};
 \foreach \from/\to in {n2/n3}
    \draw [thick, color=black](\from) -- (\to);
\foreach \from/\to in {n1/n4}
    \draw [dotted, thick,color=black](\from) -- (\to);
    \end{tikzpicture}}
\raisebox{0pt}{ $+5$}
\raisebox{-5pt}{\begin{tikzpicture}
 [scale=.15,auto=left,every node/.style={circle,fill=black!100,inner sep=0pt, minimum size=4.5pt}]
 \node (n1) at (4,9)  {};
 \node (n2) at (4,6)  {\tiny $ $};
 \node (n3) at (7,9) {\tiny $ $};
 \node (n4) at (7,6)  {\tiny $ $};
 \node (n5) at (10,9) {\tiny $ $};
 \node (n6) at (10,6) {\tiny $ $};
 \foreach \from/\to in {n2/n1}
   \draw [thick, color=black](\from) -- (\to);
    \foreach \from/\to in {}
    \draw [dotted, thick,color=black](\from) -- (\to);
    \end{tikzpicture}}
\right)  .
$$

\noindent Furthermore, if $g_1 \in \C P_{n-1,c}$ and $g_2 \in \C P_{n,c}$ define 
$$g_1g_2=\gamma(g_1)g_2.$$ 

Now that we have embedded $\C P_{n-1,c}$ into $\C P_{n,c}$,  we will show connections between the irreducible modules of $\C P_{n,c}$ and $\C P_{n-1,c}$. 
We first show that the irreducible $\C P_{n,c}$-module $W^n_T$ is a reducible $\C P_{n-1,c}$-module by constructing the following subspaces of $W^n_T$.

For each $j\in \{0,\ldots , c\}$ and $T \in X_n$ define
$$ \hat {W}^{n}_{T,j}=\C\text{-span}\{x_a ~|~ a \in P_{n,c},\beta(a)=T,\text{ and }n\in \tau_j(a)\},$$
\noindent with the usual convention that the $\C$-span of the empty set is the zero vector space. 

We know by the Artin-Weddernburn theory of semisimple algebras (Theorem \ref{artin}) that $\hat{W}^n_{T,j}$ is an irreducible representation of $\C P_{n-1,c}$ if and only if $\hat{W}^n_{T,j}$ is isomorphic to a direct summand of the regular representation of $\C P_{n-1,c}$ in Theorem \ref{main}. With this in mind, we set out to prove the following:

\begin{theorem} \label{restrict}
Fix $j \in \{0,\ldots, c\}$. Let $T \in X_n$ such that $T_j \not = \emptyset$. Then $\hat{W}^n_{T,j}$ and $W^{n-1}_{|T_0|,\ldots,|T_j|-1,\ldots,|T_c|}$ are isomorphic as $\C P_{n-1,c}$-modules.
\end{theorem}

\begin{proof}
Choose $T' \in X_{n-1}$ such that $|T'_j|=|T_j|-1$ and $|T_i'|=|T_i|$ for all $i \in \{0,\ldots,c\}\setminus \{j\}$. Because $W^{n-1}_{T'}$ is in the isomorphism class represented by $W^{n-1}_{|T_0|,\ldots,|T_j|-1,\ldots,|T_c|}$, we need only show that $\hat{W}^n_{T,j} \cong W^{n-1}_{T'}$ as $\C P_{n-1,c}$-modules.

Consider the bijection $\phi: \hat{W}^n_{T,j} \rightarrow W^{n-1}_{T'}$ given by $\phi(x_a)=x_{a'}$, where $a'$ is the unique planar diagram with $\beta(a')=T'$ and $\tau(a')=(\tau_0(a),\ldots, \tau_j(a)\setminus \{n\}, \ldots, \tau_c(a))$.

Fix $d \in P_{n-1,c}$ and $x_a \in \hat{W}^n_{T,j}$ arbitrarily. We must show that $$d\phi(x_a)=\phi(\gamma(d)x_a)=\phi \left( \sum \limits_{i=1}^{c}(d \otimes I_i)x_a-(c-1)(d\otimes I_0)x_a \right).$$

First, suppose $j \not = 0$. Fix $k \in \{0,\ldots, c\}\setminus \{j\}$. Now $n \in \beta_k(d \otimes I_k)$, so $ n \not \in \beta_j(d \otimes I_k)$. Since $n \in \tau_j(a)$, this means that $\tau_j(a) \not \subseteq \beta_j(d \otimes I_k)$. Thus, by Theorem \ref{3.1eq}, we have $(d \otimes I_k)x_a={\bf 0}$. So $$\phi(\gamma(d)x_a )= \phi((d \otimes I_j)x_a)=dx_{a'}=d\phi(x_a),$$
\noindent as desired.  

Next, suppose $j=0$. 
Observe that $(d \otimes I_{k_1})x_a=(d \otimes I_{k_2})x_a$ for all $k_1,k_2 \in \{0,\ldots, c\}$. Thus, 
\[\phi(\gamma(d)x_a)=\phi\left(\sum \limits_{i=1}^c (d\otimes I_i)x_{a}-(c-1)(d\otimes I_0)x_{a}\right)=\phi((d\otimes I_1)x_{a})=dx_{a'}= d \phi(x_a).\] 
\end{proof}

Now for $T \in X_n$, if $T_j=\emptyset$, then $\hat{W}^n_{T,j}=\{{\bf 0}\}$. Combining this observation with Theorem \ref{restrict} above, we see that 
\begin{equation} \label{direct}
W^n_{T}= \bigoplus \limits_{j=0}^c \hat{W}^n_{T,j} \cong \bigoplus \limits_{j=0}^c W^{n-1}_{|T_0|,\ldots, |T_j|-1,\ldots, |T_c|},
\end{equation}

\noindent where  for each $j\in\{0,\ldots , c\}$, we drop the summand $W^{n-1}_{|T_{0}|,\ldots , |T_{j}|-1,\ldots ,|T_{c}|}$ if $T_{j}=\emptyset$.

Equation \eqref{direct} shows that $W^n_T$ and $W^n_S$ are isomorphic as $\C P_{n-1,c}$-modules if and only if $|T_i|=|S_i|$ for all $i \in \{0,\ldots, c\}$. So, $W^n_{n_0,\ldots,n_c}$ now also denotes an isomorphism class representative of the set $\{W^n_T~|~|T_i|=n_i \text{ for all } i \in \{0,\ldots,c\}\}$ of isomorphic $\C P_{n-1,c}$-modules (see Definition \ref{defisoclass}). Theorem \ref{irred} showed that $W^{n}_{n_0, \ldots, n_c}$ is an irreducible $\C P_{n,c}$-module, and we now see it breaks up further into a direct sum of irreducible $\C P_{n-1,c}$-modules:
\begin{equation} \label{restriction} 
W^n_{n_0, \ldots, n_c}\cong \bigoplus \limits_{j=0}^c W^{n-1}_{n_0,\ldots, n_j-1,\ldots, n_c},
\end{equation}

\noindent where  for each $j\in\{0,\ldots , c\}$, we drop the summand $W^{n-1}_{n_0,\ldots, n_j-1,\ldots,n_c}$ if $n_j=0$.

We are now ready describe the Bratteli diagram of the chain $\C P_{0,c} \subseteq \C P_{1,c} \subseteq \C P_{2,c} \subseteq \cdots$ and make a connection to a common combinatorial identity for multinomial coefficients.

For a fixed $c \geq 1$, the \textit{Bratteli diagram} of the chain $\C P_{0,c} \subseteq \C P_{1,c} \subseteq \C P_{2,c} \subseteq \cdots$ is the infinite rooted tree 
with vertex set given by $\bigcup \limits_{n \geq 0} \{W^n_{n_0,\ldots,n_c} ~|~ n_0+\cdots +n_c=n\}$, the set of all irreducible representations of $\C P_{n,c}$ for all $n \geq 0$ up to isomorphism. There is an edge between vertices $W^{n}_{n_0,\ldots n_c}$ and $W^{n-1}_{m_0,\ldots , m_c}$ if and only if the latter is one of the summands in Equation \eqref{restriction} when we consider $W^{n}_{n_0,\ldots , n_c}$ as a $\C P_{n-1,c}$-module. 

In \cite{flath}, Flath, Halverson, and Herbig 
show that the dimensions of the vertices in the Bratelli diagram of $\C P_{0,1} \subseteq \C P_{1,1} \subseteq \C P_{2,1} \subseteq \cdots$ of single-colored planar rook algebras give Pascal's triangle. By coloring diagrams with two colors, we extend the results of Flath, Halverson, and Herbig and find that the dimensions of the vertices of the  Bratteli diagram of $\C P_{0,2} \subseteq \C P_{1,2} \subseteq \C P_{2,2} \subseteq \cdots$ give Pascal's tetrahedron, as shown in Figure \ref{pascaltet}. In general, the Bratteli diagram of $\C P_{0,c} \subseteq \C P_{1,c} \subseteq \C P_{2,c} \subseteq \cdots$ yields Pascal's $(c+1)$-dimensional simplex. 

\begin{figure}[h] \caption{The Bratteli diagram of $\C P_{0,2} \subseteq \C P_{1,2} \subseteq \C P_{2,2} \subseteq \cdots$ gives Pascal's tetrahedron, the first three levels of which are drawn below.}
\begin{center} \label{pascaltet}
\begin{tikzpicture}
[scale=.30,auto=left,every node/.style={circle, inner sep=-2.5pt,minimum size=5pt}]
  
  \node (n2) at (13.5,24.5) {\tiny{${W^2_{2,0,0}}$}};
  \node (n5) at (8,16.5) {\tiny{${W^2_{1,1,0}}$}};
  \node (n6) at (19,16.5) {\tiny{${W^2_{1,0,1}}$}};  
  \node (m3) at (22,15.5) {};
  \node (n10) at (4.5,7.5) {\tiny{${W^2_{0,2,0}}$}};
  \node (n11) at (13.5,7.5) {\tiny{${W^2_{0,1,1}}$}};  
  \node (n12) at (22.5,7.5) {\tiny{${W^2_{0,0,2}}$}};
    \node (n17) at (13.5,20) {\tiny{$W^1_{1,0,0}$}};
    \node (n18) at (9,11) {\tiny{$W^1_{0,1,0}$}};
    \node(n19) at (18,11) {\tiny{$W^1_{0,0,1}$}};
    \node(n20) at (13.5,14.5) {\tiny{$W^0_{0,0,0}$}};         
   
\foreach \from/\to in
{n17/n2,n17/n5,n17/n6,n19/n6,n19/n12,n19/n11,n18/n5,n18/n10,n18/n11,n20/n17,n20/n19,n20/n18} 
\draw [thick](\from) -- (\to); 

\end{tikzpicture}
\end{center}
\end{figure}
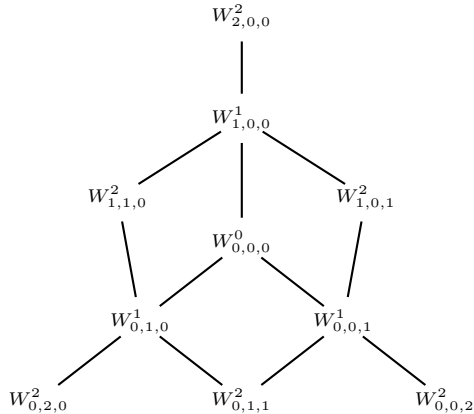

We note that not only do elements of Pascal's simplex appear in our Bratteli diagrams, but the edges of our Bratelli diagrams indicate the recursive relationship between vertices in the $n^{th}$ and $(n-1)^{th}$ level of Pascal's simplex. Taking the dimensions of both sides of Equation \eqref{restriction}, we find another proof of the following well known combinatorial identity for multinomial coefficients:

$$\binom{n}{n_0,n_1, \ldots, n_c}=\binom{n-1}{n_0-1,n_1, \ldots,n_c}+\binom{n-1}{n_0,n_1-1, \ldots, n_c}+\cdots +\binom{n-1}{n_0,n_1, \ldots, n_c-1}.$$
 
\

For large values of $c$, the Bratteli diagram of $\C P_{0,c} \subseteq \C P_{1,c} \subseteq \C P_{2,c} \subseteq \cdots $ can be difficult to visualize. Thus we now offer a description of its structure. 

Let the $n^{th}$ level of the Bratelli diagram of  $\C P_{0,c} \subseteq \C P_{1,c} \subseteq \C P_{2,c} \subseteq \cdots $ be the set of all vertices of the form $W^n_{n_0,\ldots, n_c}$. 
The number of vertices in the $n^{th}$ level is equal to 
$|\{W^n_{n_0,\ldots,n_c}~|~ n_0+\cdots + n_c=n\}|$, which is equal to ${n+c \choose n}$,  the number of nonnegative integer solutions to $n_0 + \cdots + n_c =n$. Figure \ref{bc} gives this number for various values of $n$ and $c$. Notice the presence of the rows of Pascal's triangle along the skew-diagonal in the given table.

\begin{figure}[h]
\caption{ The number of vertices in the $n^{th}$ level of the Bratteli diagram of $\C P_{0,c} \subseteq \C P_{1,c} \subseteq \C P_{2,c} \subseteq \cdots$ gives rise to Pascal's triangle.} 
\label{bc}

$$\begin{array}{|c||c|c|c|c|c|c|c|}
\hline
c=&0&1&2&3&4&\cdots \\
\hline
\hline
n=0 & 1&1&1&1&1&\cdots\\
\hline
n=1&1 & 2&3&4&5&\cdots\\
\hline
n=2&1&3&6&10&15&\cdots\\
\hline
n=3&1&4 &10&20&35&\cdots\\
\hline
n=4&1&5&15&35&70&\cdots\\
\hline
\vdots & \vdots&\vdots &\vdots &\vdots & \vdots&\vdots\\
\hline
\end{array}$$
\end{figure}
The following proposition gives information about the connectivity of a given Bratteli diagram.

\begin{proposition}
Let $n,x\geq 1$. In the Bratteli diagram of $\C P_{0,c} \subseteq \C P_{1,c} \subseteq \C P_{2,c} \subseteq \cdots$ the number of of vertices in the $n^{th}$ level adjacent to exactly $x$ vertices in the $(n-1)^{th}$ level is ${c+1 \choose x}{n-1 \choose x-1}$.
\end{proposition}

\begin{proof}  
The vertex $W^n_{n_0,\ldots, n_c}$
is adjacent to $x$ vertices if and only if the number of nonzero entries in $\{n_0,\ldots,n_c\}$ is $x$. There are ${c+1 \choose x}$ ways to choose $x$ of the $n_i$ to be nonzero. Next we count the number of ways to assign values to the $x$ chosen $n_i$ such that their sum equals $n$. This is equivalent to counting the number of positive integer solutions to $n_{i_1}+n_{i_2}+\cdots + n_{i_x}=n$, which is ${n-1 \choose x-1}$. Thus the total number of vertices in the $n^{th}$ level that are adjacent to $x$ vertices in the $(n-1)^{th}$ level is ${c+1 \choose x}{n-1\choose x-1}$.
\end{proof}

\bigskip

\section{Character Table} \label{section character}

The character of $W^n_{n_0,\ldots, n_c}$, denoted by $\chi_{n_0,\ldots,n_c}^n$, is the $\C$-valued function that gives the trace of each $d \in P_{n,c}$ as a linear operator on $W_{n_0,\ldots,n_c}^n$. In \cite{flath}, Flath, Halverson, and Herbig show that the character of each irreducible representation of $\C P_{n,1}$ is a binomial coefficient. In this section, we find the character of each $W^n_{n_0,\ldots, n_c}$ and show that it is the product of $c$ binomial coefficients.

We say that an edge connecting bottom vertex $i$ to top vertex $j$ is \textit{vertical} if $i=j$.
Below we will show that the problem of computing the trace of digram $d$ can be reduced to considering the subdiagram of $d$ containing only the vertical edges of $d$.
 
\begin{proposition}\label{treq}
Let $d \in P_{n,c}$. Then Tr$(d)=$ Tr$(d_v)$, where $d_v$ is the subdiagram of $d$ resulting from removing all non-vertical edges in $d$.
\end{proposition}
\begin{proof}
Suppose $d \in P_{n,c}$ has at least one non-vertical edge colored $u_i$. Then because $d$ is planar there exists an edge $(j,k)$ colored $u_i$ with bottom vertex $j$ and top vertex $k$ such that $j \notin \tau_i(d)$. (If $d$ contains positive sloping edges colored $u_i$, take $(j,k)$ to be the left most such edge. Otherwise, take $(j,k)$ to be the right most negative sloping edge colored $u_i$.)
Let $e_d \in P_{n,c}$ be the diagram such that $\tau(e_d) = \beta(e_d) = \tau(d)$. Now $e_dd = d$, and hence, by the commutativity of the trace function,
$$Tr(d) = Tr(e_dd) = Tr(de_d).$$ Moreover, every vertical edge in $d$ is also a vertical edge in $de_d$. Furthermore, $j \notin \tau_i(d)= \beta_i(e_d)$ implies $j \notin \beta_i(de_d)$, and thus, $(j,k)$ is not an edge in $de_d$. Iterating this process to remove non-vertical edges in $d$, we see that Tr$(d)=$ Tr$(d_v)$, where $d_v$ is the subdiagram of $d$ resulting from removing all non-vertical edges in $d$.
\end{proof}

\begin{theorem} \label{character}
For $d \in P_{n,c}$ the value of the irreducible character is given by
\[\chi_{n_0,\ldots, n_c}^n(d)=\begin{cases}
\prod\limits_{i=1}^c {\ell_i\choose n_i} & \text{ if } n_i \leq \ell_i \text{ for all } i \in \{1,\ldots, c\}, \\
0 & \text{otherwise,}
\end{cases}\]
where $\ell_i$ is the number of vertical edges in $d$ colored $u_i$.
\end{theorem}
\begin{proof}
Fix $d \in P_{n,c}$. By Proposition \ref{treq}, for our purposes we may assume $d$ has only vertical edges.  

Now for $x_a \in W^n_{n_0,\ldots, n_c}$, either $dx_a={\bf 0}$ or $dx_a=x_{da}$. Thus, the $x_a$-$x_a$ entry of the matrix of $d$ as a linear operator on $W^n_{n_0,\ldots, n_c}$ will be $1$ if $dx_a=x_{da}=x_a$ and 0 otherwise. Therefore, 
\[\chi^n_{n_0,\ldots, n_c}(d)= |\{x_a ~|~ x_a \in W^n_{n_0,\ldots, n_c}, \text{ and } dx_{a}=x_a\}|.\]

That is, we must count the number of $x_a \in W^n_{n_0,\ldots, n_c}$ such that $da=a$. Let $i \in \{1,\ldots, c\}$. Since all lines in $d$ are vertical, $|\tau_i(a)|=n_i$, and $|\beta_i(d)|=\ell_i$, it follows that there are $ {\ell_i \choose n_i}$ possible ways to choose $\tau_i(a)$, where we let ${\ell_i \choose n_i}=0$ if $n_i >\ell_i$. (Note that there is only one option for $\beta(a)$, since $x_a \in W^n_{n_0,\ldots,n_c}$ implies that $a$ must have some fixed bottom set.) Because this is true for each $i \in \{1,\ldots, c\}$ and the choices for each $i$ are independent, we have 
\[\chi_{n_0,\ldots, n_c}^n(d)=\prod\limits_{i=1}^c {\ell_i\choose n_i}.\]
\end{proof}

 \bigskip
\section{Acknowledgments}

The authors were supported by an NSF-sponsored REU grant provided to the University of California Santa Barbara during the writing of this article.

Many thanks to our advisor, Stephen Bigelow, who presented us with the idea for this paper and who was a tremendously helpful resource throughout the project.

We are also grateful to Anthony Caine and Kevin Malta, who contributed in the early stages of the project, particularly in defining multiplication of the multi-colored planar rook semigroup.

\bibliographystyle{plain}
\bibliography{Bibliography}{}

\end{document}